\documentclass[twoside,11pt, reqno]{amsart}

\usepackage{a4wide}
\usepackage{latexsym}
\usepackage{mathrsfs}
\usepackage[T1]{fontenc}
\usepackage{enumitem}

\usepackage{amssymb}
\usepackage{latexsym}
\usepackage{amsmath}
\usepackage{fancyhdr}
\usepackage{bbm}

\numberwithin{equation}{section}

\newtheorem{theorem}{Theorem}[section] 
\newtheorem{lemma}[theorem]{Lemma}     
\newtheorem{corollary}[theorem]{Corollary}

\theoremstyle{definition}

\newcommand {\C}{\mathbb C}

\newcommand {\N}{\mathbb N}

\newcommand{\len}{{\rm len \,}}

\newcommand{\spn}{{\rm span \,}}
\newcommand{\1}{\mathbbm{1}}

\begin{document}

\title{An Infinite C*-algebra with a Dense, Stably Finite *-subalgebra}

\author[N.\ J.\  Laustsen]{Niels Jakob Laustsen}
\author[J.\ T.\  White]{Jared T.\ White}
\address{
Niels Jakob Laustsen/Jared T. White\\
Department of  Mathematics and Statistics\\
Lancaster University\\
Lancaster LA1 4YF\\
United Kingdom}
\email{n.laustsen@lancaster.ac.uk, j.white6@lancaster.ac.uk}



\begin{abstract}
We construct a unital pre-C*-algebra $A_0$ which is stably finite, in the sense that every left invertible square matrix over $A_0$ is right invertible, while the C*-completion of $A_0$ contains a non-unitary isometry, and so it is infinite.

\medskip
\noindent
To appear in \emph{Proceedings of the American Mathematical Society}.
\end{abstract}

\maketitle

\section{Introduction}
\noindent 
Let $A$ be a unital algebra. We say that $A$ is \textit{finite} (also called directly finite or Dedekind finite) if every left invertible element of $A$ is right invertible, and we say that $A$ is \textit{infinite} otherwise. This notion originates in the seminal studies of projections in von
Neumann algebras carried out by Murray and von Neumann in the 1930s.
At the $22^{\text{nd}}$ International Conference on Banach Algebras and Applications, held at the Fields Institute in Toronto in 2015, Yemon Choi raised the following questions:
\begin{enumerate}
\item \label{q1} Let $A$ be a unital, finite normed algebra. Must its completion be finite?
\item \label{q2} Let $A$ be a unital, finite pre-C*-algebra. Must its completion be finite?
\end{enumerate}
\noindent Choi also stated Question \eqref{q1} in \cite[Section 6]{C}.

A unital algebra $A$ is said to be \textit{stably finite} if the matrix algebra $M_n(A)$ is finite for each  $n \in \N$. This stronger form of finiteness is particularly useful in the
context of $K$-theory, and so it has become a household item in the Elliott
classification programme for C*-algebras. The notions of finiteness
and stable finiteness differ even for C*-algebras, as was shown
independently by Clarke \cite{NPC} and Blackadar~\cite{Bl1} (or see \cite[Exercise 6.10.1]{Bl2}).
A much deeper result is due to R\o{}rdam
\cite[Corollary~7.2]{rordam}, who constructed a unital, simple
C*-algebra which is finite (and separable and nuclear), but not
stably finite.

We shall answer Question \eqref{q2}, and hence Question \eqref{q1}, in the negative by proving the following result:

\begin{theorem} \label{-1} 
There exists a unital, infinite C*-algebra which contains a dense, unital, stably finite *-subalgebra.
\end{theorem}

Let $A$ be a unital *-algebra. Then there is a natural variant of finiteness in this setting, namely we say that $A$ is \textit{*-finite} if whenever we have $u \in A$ satisfying $u^* u = 1$, then $uu^* = 1$. However, it is known (see, \emph{e.g.}, \cite[Lemma 5.1.2]{RLL}) that a C*-algebra is finite if and only if it is *-finite, so in this article we shall not need to refer to *-finiteness again.

\section{Preliminaries}
\noindent
Our approach is based on semigroup algebras. Let $S$ be  a monoid, that is, a semigroup with an identity, which we shall usually denote by $e$. By an \textit{involution} on $S$ we mean a map from $S$ to $S$, always denoted by $s \mapsto s^*$, satisfying $(st)^* = t^* s^*$ and $s^{**} = s \; (s, t \in S)$.  By a \textit{*-monoid} we shall mean a pair $(S, *)$, where $S$ is a monoid, and $*$ is an involution on $S$. Given a *-monoid $S$, the semigroup algebra $\C S$ becomes a unital *-algebra simply by defining $\delta_s^* = \delta_{s^*} \; (s \in S)$, and extending conjugate-linearly. 

Next we shall recall some basic facts about free products of *-monoids, unital *-algebras, and their C*-representations.

Let $S$ and $T$ be monoids, and let $A$ and $B$ be unital algebras. Then we denote the free product (\emph{i.e.} the coproduct) of $S$ and $T$ in the category of monoids by $S*T$, and similarly we denote the free product of the unital algebras $A$ and $B$ by $A*B$. It follows from the universal property satisfied by free products that, for monoids $S$ and $T$, we have $\C(S*T) \cong (\C S)*(\C T)$.

Given *-monoids $S$ and $T$, we can define an involution on $S*T$ by 
$$(s_1t_1 \cdots s_nt_n)^* = t_n^*s_n^* \cdots t_1^* s_1^* $$
for $n \in \N, s_1 \in S, s_2, \ldots, s_n \in S \setminus \{ e \}, t_1, \ldots, t_{n-1} \in T \setminus \{ e \},$ and $t_n \in T.$  The resulting *-monoid, which we continue to denote by $S*T$, is the free product in the category of *-monoids. We can analogously define an involution on the free product of two unital *-algebras, and again the result is the free product in the category of unital *-algebras. We then find that $\C(S*T) \cong (\C S)*(\C T)$ as unital *-algebras.

Let $A$ be a *-algebra. If there exists an injective *-homomorphism from  $A$  into some C*-algebra, then we say that $A$ admits a \textit{faithful C*-representation}. In this case, $A$ admits a norm such that the completion of $A$ in this norm is a C*-algebra, and we say that $A$ admits a \textit{C*-completion}. Our construction will be based on C*-completions of *-algebras of the form $\C S$, for $S$ a *-monoid. 

We shall denote by $S_\infty$ the free *-monoid on countably many generators; that is, as a monoid $S_\infty$ is free on some countably-infinite generating set ${\lbrace t_n, s_n : n \in \N \rbrace}$, and the involution is determined by $t_n^* = s_n \: (n \in \N)$. For the rest of the text we shall simply write $t_n^*$ in place of $s_n$. We define $BC$ to be the bicyclic monoid $\langle  p,q : pq = e \rangle$. This becomes a *-monoid when an involution is defined by $p^* = q$, and the corresponding $*$-algebra~$\C{}BC$ is infinite because $\delta_p\delta_q = \delta_e$, but $\delta_q\delta_p = \delta_{qp} \neq\delta_e.$

\begin{lemma} \label{6a} 
The following unital *-algebras admit faithful C*-representa\-tions:
\begin{enumerate}
\item[{\rm (i)}] $\C(BC),$
\item[{\rm (ii)}] $\C(S_\infty)$.
\end{enumerate}
\end{lemma}

\begin{proof}
(i) Since $BC$ is an inverse semigroup, this follows from \cite[Theorem 2.3]{B}. 

(ii) By \cite[Theorem 3.4]{BD} $\C S_2$ admits a faithful C*-representation, where $S_2$ denotes the free monoid on two generators $S_2 = \langle a, b \rangle$, endowed with the involution determined by $a^* = b$. There is a *-monomorphism $S_\infty \hookrightarrow S_2$ defined by 
$t_n \mapsto a(a^*)^na \ (n \in \N)$
and this induces a *-monom\-orphism $\C S_\infty \hookrightarrow \C S_2$. The result follows.
\end{proof}

By a \textit{state} on a unital *-algebra $A$ we mean a linear functional $\mu \colon A \rightarrow \C$ satisfying $\langle a^*a, \mu \rangle \geq 0 \ (a \in A)$ and $\langle 1, \mu \rangle = 1$. We say that a state $\mu$ is \textit{faithful} if $\langle a^*a, \mu \rangle >0 \ (a \in A \setminus \{ 0 \})$. A unital *-algebra with a faithful state admits a faithful C*-representation via the GNS representation associated with the state.

The following theorem appears to be folklore in the theory of free products of $C^*$-algebras; it can be traced back at least to the seminal work of Avitzour~\cite[Proposition~2.3]{A} (see also \cite[Section~4]{BlPaul} for 
a more general result).  

\begin{theorem} \label{7} 
Let $A$ and $B$ be unital $*$-algebras which admit faithful
states. Then their free product $A*B$ also admits a faithful
state, and hence it has a faithful C*-re\-pre\-sen\-ta\-tion.
\end{theorem}
\noindent
We make use of this result in our next lemma.

\begin{lemma} \label{6} 
The unital *-algebra $\C(BC*S_\infty)$ admits a faithful C*-repres\-entation.
\end{lemma}

\begin{proof}
We first remark that a separable C*-algebra $A$ always admits a faithful state. To see this, note that the unit ball of $A^*$ with the weak*-topology is a compact metric space, and hence also separable. It follows that the set of states $S(A)$ is weak*-separable. Taking $\{ \rho_n : n \in \N \}$ to be a dense subset of $S(A)$, we then define $\rho = \sum_{n=1}^\infty 2^{-n} \rho_n$, which is easily seen to be a faithful state on $A$. 

By Lemma \ref{6a}, both $\C(BC)$ and $\C(S_\infty)$ admit C*-completions. Since both of these algebras have countable dimension, their C*-completions are separable, and, as such, each admits a faithful state, which we may then restrict to obtain faithful states on $\C BC$ and $\C S_\infty$. By Theorem \ref{7}, $(\C BC)*(\C S_\infty) \cong {\C(BC*S_\infty)}$ admits a faithful C*-representation.
\end{proof}

\section{Proof of Theorem \ref{-1}}
\noindent
The main idea of the proof is to embed $\C S_\infty$, which is finite, as a dense *-subalgebra of some C*-completion of $\C(BC*S_\infty)$, which will necessarily be infinite. In fact we have the following:
\begin{lemma} \label{0} 
The *-algebra $\C S_\infty$ is stably finite.
\end{lemma}

\begin{proof}
As we remarked in the proof of Lemma \ref{6a}, $\C S_\infty$ embeds into $\C S_2$. It is also clear that, as an algebra, $\C S_2$ embeds into $\C F_2$, where $F_2$ denotes the free group on two generators. Hence $\C S_\infty$ embeds into $\operatorname{vN}(F_2)$, the group von Neumann algebra of $F_2$, which is stably finite since it is a C*-algebra with a faithful tracial state. It follows that $\C S_\infty$ is stably finite as well.
\end{proof}

We shall next define a notion of length for elements of $BC*S_\infty$. Indeed, each $u \in {(BC*S_\infty) \setminus \{e \}}$ has a unique expression of the form $w_1w_2 \cdots w_n$, for some $n \in \N$ and some 
$w_1, \ldots, w_n \in \left(BC \setminus \lbrace e \rbrace \right) \cup \lbrace t_j,  t_j^* :j \in \N \rbrace,$
satisfying $w_{i+1} \in \lbrace t_j,  t_j^* :j \in \N \rbrace$ whenever $w_i \in  BC \setminus \lbrace e \rbrace \ (i=1, \ldots, n-1)$.
We then define $\len u =n $ for this value of $n$, and set $\len e = 0$. 
This also gives a definition of length for elements of $S_\infty$ by considering $S_\infty$ as a submonoid of $BC*S_\infty$ in the natural way. For $m \in \N_0$ we set
$$L_m(BC*S_\infty) = \{ u \in BC*S_\infty : \len u \leq m \}, \quad
L_m(S_\infty) = \{ u \in S_\infty : \len u \leq m \}.$$

We now describe our embedding of $\C S_\infty$ into $\C (BC*S_\infty$). By Lemma \ref{6}, $\C(BC*S_\infty)$ has a C*-completion $(A, \Vert \cdot \Vert)$. Let $\gamma_n = (n \Vert \delta_{t_n} \Vert)^{-1} \ (n \in \N)$ and define elements $a_n$ in ${\C (BC*S_\infty)}$ by $a_n = \delta_p + \gamma_n \delta_{t_n} \ (n \in \N)$,  so that $a_n \rightarrow \delta_p$ as $ n \rightarrow \infty$. Using the universal property of $S_\infty$ we may define a unital *-homomorphism $\varphi \colon \C S_\infty  \rightarrow  \C (BC*S_\infty)$ by setting $\varphi(\delta_{t_n}) = a_n \ (n \in  \N)$ and extending to $\C S_\infty$.  In what follows, given a monoid $S$ and $s \in S$, $\delta_s'$ will denote the linear functional on $\C S$ defined by 
$\langle \delta_t, \delta_s' \rangle = \1_{s, t} \quad (t \in S),$
 where $\1_{s, t} = 1$ if $s=t$ and $\1_{s, t} = 0$ otherwise. 

\begin{lemma} \label{1} 
Let $w \in S_\infty$ with $\len w = m$. Then 
\begin{enumerate}
\item[\rm (i)] $\varphi(\delta_w) \in \spn \{ \delta_u :  u \in L_m(BC*S_\infty) \}$;
\item[\rm (ii)] for each $y \in L_m(S_\infty)$ we have
$$\langle \varphi(\delta_y), \delta_w' \rangle \neq 0 \Leftrightarrow y = w.$$
\end{enumerate}
\end{lemma}

\begin{proof}
We proceed by induction on $m$. When $m =0$, $w$ is forced to be $e$ and hence, as $\varphi$ is unital, $\varphi(\delta_e)= \delta_e$, so that (i) is satisfied. In (ii), $y$ is also equal to $e$, so that (ii) is trivially satisfied as well. 

Assume $m \geq 1$ and that (i) and (ii) hold for all elements of $L_{m-1}(S_\infty)$. We can write $w$ as $w=vx$ for some $v \in S_\infty$ with $\len v = m-1$ and some $x \in \{ t_j, t_j^* : j \in \N \}$. 

First consider (i). By the induction hypothesis, we can write $\varphi( \delta_v) = \sum_{u \in E} \alpha_u \delta_u$, for some finite set $E \subset L_{m-1}(BC*S_\infty)$ and some scalars  $\alpha_u \in \C \ (u \in E)$. Suppose that $x = t_j$ for some $j \in \N$. Then
$$\varphi(\delta_w) = \varphi(\delta_v) \varphi(\delta_{t_j}) = \left( \sum_{u \in E} \alpha_u \delta_u \right)(\delta_p + \gamma_j\delta_{t_j})
= \sum_{u \in E} \alpha_u \delta_{up} + \alpha_u \gamma_j \delta_{ut_j},$$
which belongs to $\spn \{ \delta_u :  u \in L_m(BC*S_\infty) \}$ because 
$$\len(up) \leq \len(u) + 1 \leq  m \quad \text{and}  \quad
\len(ut_j) = \len(u) +1 \leq m$$
for each $u \in L_{m-1}(BC*S_\infty)$. The case $x=t_j^*$ is established analogously.

Next consider (ii). Let $y \in L_m(S_\infty)$. If $\len y \leq m-1$ then, by (i), we know that $\varphi(\delta_y) \in \spn \{\delta_u : u \in  L_{m-1}(BC*S_\infty) \} \subset \ker \delta_w'$. Hence in this case $y \neq w$ and $\langle \varphi(\delta_y), \delta_w' \rangle = 0$.

Now suppose instead that $\len y = m$, and write $y = uz$ for some ${u \in L_{m-1}(S_\infty)}$ and $z \in  \{ t_j, t_j^*: j \in \N \}$. By (i) we may write $\varphi(\delta_u) = \sum_{s \in F} \beta_s \delta_s$ for some finite subset $F \subset L_{m-1}(BC*S_\infty)$ and some scalars $\beta_s \in \C \ (s \in F)$, and we may  assume that $v \in F$ (possibly with $\beta_v = 0$). We prove the result in the case that $z = t_j$ for some $j \in  \N$, with the argument for the case $z=t_j^*$ being almost identical. We have $\varphi(\delta_z) = \delta_p+\gamma_j \delta_{t_j}$ and it follows that 
$$\varphi(\delta_y) = \varphi(\delta_u) \varphi(\delta_z) = \sum_{s  \in F} \beta_s \delta_{sp} + \beta_s\gamma_j\delta_{st_j}.$$
Observe that $sp \neq w$ for each $s \in F$.
This is because we either have $\len(sp) < m = \len(w)$, or else $sp$ ends in $p$ when considered as a word over the alphabet $\{p, p^* \} \cup \lbrace t_j, t_j^* :j \in \N \rbrace$, whereas $w \in S_\infty$.
Moreover, given $s \in F$, $st_j = w = vx$ if and only if $s = v$ and $t_j = x$. Hence
$$ \langle \varphi(\delta_y), \delta_w' \rangle = \beta_v\gamma_j\1_{t_j, x} = \langle \varphi(\delta_u), \delta_v' \rangle \gamma_j \1_{t_j, x}.$$
As $\gamma_j > 0$, this implies that $\langle \varphi(\delta_y), \delta_w' \rangle \neq 0$ if and only if $\langle \varphi(\delta_u), \delta_v' \rangle \neq 0$ and $t_j = x$, which, by the induction hypothesis, occurs if and only if $u=v$ and $t_j = x$. This final statement is equivalent to $y=w$. 
\end{proof}

\begin{corollary} \label{2} 
The map $\varphi$ is injective.
\end{corollary}

\begin{proof}
Assume towards a contradiction that $\sum_{u \in F} \alpha_u \delta_u \in \ker \varphi$ for some non-empty finite set $F \subset S_\infty$ and $\alpha_u \in \C \setminus \{0\} \ (u \in F)$. Take $w \in  F$ of maximal length. Then
\begin{align*}
0 = \left \langle \varphi \left(\sum_{u \in F} \alpha_u \delta_u \right), \delta_w' \right\rangle = \sum_{u \in F} \alpha_u \langle \varphi(\delta_u), \delta_w' \rangle 
= \alpha_w \langle \varphi(\delta_w), \delta_w' \rangle,
\end{align*}
where the final equality follows from Lemma \ref{1}(ii). That lemma also tells us that $\langle \varphi(\delta_w), \delta_w' \rangle \neq 0$, forcing $\alpha_w = 0$, a contradiction.
\end{proof}

We can now prove our main theorem.

\begin{proof}[Proof of Theorem \ref{-1}]
Recall that $(A, \Vert \cdot \Vert)$ denotes a C*-completion of $\C(BC*S_\infty)$, which exists by Lemma \ref{6}, and $A$ is infinite since $\delta_p, \delta_q \in A$. Let $A_0 \subset A$ be the image of $\varphi$. Corollary \ref{2} implies that $A_0 \cong \C S_\infty$, which is stably finite by Lemma \ref{0}.
Moreover, $\varphi(\delta_{t_n}) = a_n \rightarrow \delta_p$ as $n \rightarrow \infty$, so that $\delta_p \in \overline{A_0},$ and we see also that $\delta_{t_n} = \frac{1}{\gamma_n}(a_n - \delta_p) \in \overline{A_0} \ (n \in \N).$ The elements $\delta_p$ and $\delta_{t_n} \ (n \in \N)$ generate $A$ as a C*-algebra, and since $\overline{A_0}$ is a C*-subalgebra containing them, we must have $A= \overline{A_0}$, which completes the proof.
\end{proof}

\subsection*{Acknowledgements}
\noindent 
We are grateful to Yemon Choi for several helpful discussions whilst working on this problem. We are also grateful to the referee for his/her useful comments, in particular those concerning the references for Theorem \ref{7}.

\end{document}